\documentclass [12pt]{amsart}
\usepackage{amsmath,amssymb}
 \usepackage{amsthm, amsfonts}
 \usepackage{enumerate}

\newtheorem{theorem}{Theorem}[section]

\newtheorem{lemma}[theorem]{Lemma}
\newtheorem{corollary}[theorem]{Corollary}
\theoremstyle{definition}
\newtheorem{definition}[theorem]{Definition}
\newtheorem{example}[theorem]{Example}

\theoremstyle{remark}

\numberwithin{equation}{section}

\begin{document}

\title [{Some properties of graded generalized 2-absorbing submodules}]{Some properties of graded generalized 2-absorbing submodules }

 \author[{{S. Alghueiri and K. Al-Zoubi }}]{\textit{Shatha Alghueiri and Khaldoun Al-Zoubi*  }}

\address
{\textit{Shatha Alghueiri, Department of Mathematics and Statistics,
Jordan University of Science and Technology, P.O.Box 3030, Irbid
22110, Jordan.}}
\bigskip
{\email{\textit{ghweiri64@gmail.com}}}
\address
{\textit{Khaldoun Al-Zoubi, Department of Mathematics and
Statistics, Jordan University of Science and Technology, P.O.Box
3030, Irbid 22110, Jordan.}}
\bigskip
{\email{\textit{kfzoubi@just.edu.jo}}}


\date{}
\begin{abstract}
Let $G$ be an abelian group with identity $e$. Let $R$ be a $G$-graded
commutative ring and $M$ a graded $R$-module. In this paper we will obtain some results concerning the graded generalized
2-absorbing submodules and their homogeneous components. Special attention has been paid, when graded rings are graded gr-Noetherian, to find extra properties of these graded submodules.
\end{abstract}

\keywords{graded generalized 2-absorbing submodule, graded 2-absorbing submodule, graded 2-absorbing primary submodule  . \\
$*$ Corresponding author}
 \maketitle
\section{Introduction and Preliminaries}
Throughout this paper all rings are commutative with identity and all
modules are unitary.

Badawi in \cite{18} introduced the concept of 2-absorbing ideals of
commutative rings. The notion of 2-absorbing ideals was extended to
2-absorbing submodules in \cite{20} and \cite{29}. The concept of
generalized 2-absorbing submodules, as a generalization of 2-absorbing
submodules, was introduced in \cite{22}.

In \cite{30}, Refai and Al-Zoubi introduced the concept of graded
primary ideal. The concept of graded 2-absorbing ideals, as a generalization
of graded prime ideals, was introduced and studied by Al-Zoubi, Abu-Dawwas
and Ceken in\cite{5}.
In \cite{14}, Al-Zoubi and Sharafat introduced the concept of graded
2-absorbing primary ideal which is a generalization of graded primary ideal.
The concept of graded prime submodules was introduced and studied by many
authors, see for example \cite{2, 3, 11, 12, 13, 14, 15, 28}.
The concept of graded 2-absorbing submodules, as a generalization of graded
prime submodules, was introduced by Al-Zoubi and Abu-Dawwas in \cite{4}
and studied in \cite{8, 9}. Then many generalizations of graded
2-absorbing submodules were studied such as graded 2-absorbing primary (see
\cite{19}), graded weakly 2-absorbing primary (see \cite{7} ) and
graded classical 2-absorbing submodule (see \cite{6}).

Recently, the authors in \textbf{[1], }introduced the concept of graded
generalized 2-absorbing submodules as a generalization of graded 2-absorbing
submodules.

Here, we introduce several results concerning graded generalized 2-absorbing submodules and their homogeneous components.

First, we recall some basic properties of graded rings and modules which
will be used in the sequel. We refer to \cite{24, 25, 26, 27} for
these basic properties and more information on graded rings and modules.
Let $G$ be an abelian multiplicative group with identity element $e.$ A ring
$R$ is called a graded ring (or $G$-graded ring) if there exist additive
subgroups $R_{h}$ of $R$ indexed by the elements $h\in G$ such that $%
R=\oplus _{h\in G}R_{h}$ and $R_{g}R_{h}\subseteq R_{gh}$ for all $g,h\in G$%
. The non-zero elements of $R_{h}$ are said to be homogeneous of degree $h$
and all the homogeneous elements are denoted by $h(R)$, i.e. $h(R)=\cup
_{h\in G}R_{h}$. If $a\in R$, then $a$ can be written uniquely as $%
\sum_{h\in G}a_{h}$, where $a_{g}$ is called a homogeneous component of $a$
in $R_{h}$. Moreover, $R_{e}$ is a subring of $R$ and $1\in R_{e}$. Let $R=\oplus _{h\in G}R_{h}$ be a $G$-graded ring. An ideal
$J$ of $R$ is said to be a graded ideal if $J=\sum_{h\in G}(J\cap
R_{h}):=\sum_{h\in G}J_{h}$, (see \cite{27}).

Let $R=\oplus _{h\in G}R_{h}$ be a $G$-graded ring. A left $R$-module $M$ is
said to be a graded $R$-module (or $G$-graded $R$-module) if there exists a
family of additive subgroups $\{M_{h}\}_{h\in G}$ of $M$ such that $M=\oplus
_{h\in G}M_{h}$ and $R_{g}M_{h}\subseteq M_{gh}$ for all $g,h\in G$. Also if
an element of $M$ belongs to $\cup _{h\in G}M_{h}=h(M)$, then it is called a
homogeneous. Note that $M_{h}$ is an $R_{e}$-module for every $h\in G$. Let $%
R=\oplus _{h\in G}R_{h}$ be a $G$-graded ring. A submodule $N$ of $M$ is
said to be a graded submodule of $M$ if $N=\oplus _{h\in G}(N\cap
M_{h}):=\oplus _{h\in G}N_{h}$. In this case, $N_{h}$ is called the $h$%
-component of $N$, (see \cite{27}).

Let $R$ be a $G$-graded ring and $M$ a graded $R$-module. The \textit{graded radical} of a graded
ideal $I$, denoted by $Gr(I)$, is the set of all $\ x=\sum_{g\in G}x_{g}\in R
$ such that for each $g\in G$ there exists $n_{g}>0$ with $x_{g}^{n_{g}}\in I
$. Note that, if $r$ is a homogeneous element, then $r\in Gr(I)$ if and only
if $r^{n}\in I$ for some $n\in \mathbb{N}$,(see \cite{30}).
The \textit{graded radical} of a graded submodule $N$ of $M$,
denoted by $Gr_{M}(N)$, is defined to be the intersection of all graded
prime submodules of $M$ containing $N$. If $N$ is not contained in any
graded prime submodule of $M$, then $Gr_{M}(N)=M$ (see \cite{16}).



 \section{Graded $G2$-absorbing submodules}

\begin{definition}\cite[Definition 2.2]{1}
 Let $R$ be a $G$-graded ring, $M$ a graded $R$-module and $N$ a proper
graded submodule of $M$. Then $N$ is said to be a graded generalized
2-absorbing or graded $G2$-absorbing\ submodule of $M$ if whenever $%
r_{h},s_{\alpha }\in h(R)$ and $m_{\lambda }\in h(M)$ with $r_{h}s_{\alpha
}m_{\lambda }\in N$, implies either $r_{h}^{k_{1}}\in (N:_{R}m_{\lambda })$
or $s_{\alpha }^{k_{2}}\in (N:_{R}m_{\lambda })$ or $r_{h}s_{\alpha }\in
(N:_{R}M)$, for some $k_{1},k_{2}\in
\mathbb{Z}
^{+}.$
\end{definition}
\begin{definition}
Let $R$ be a $G$-graded ring, $M$ a
graded $R$-module, $N=\oplus _{g\in G}N_{g}$ a graded submodule of $M$ and $%
g\in G.$ We say that $N_{g}$ is a $g$-generalized 2-absorbing or $g$-$G2$%
-absorbing submodule of the $R_{e}$-module $M_{g}$ if $N_{g}\not=M_{g};$ and
whenever $r_{e},s_{e}\in R_{e}$ and $m_{g}\in M_{g}$ with $%
r_{e}s_{e}m_{g}\in N_{g}$, implies either $r_{e}^{k_{1}}\in
(N_{g}:_{R_{e}}m_{g})$ or $s_{e}^{k_{2}}\in (N_{g}:_{R_{e}}m_{g})$ or $%
r_{e}s_{e}\in (N_{g}:_{R_{e}}M_{g})$, for some $k_{1},$ $k_{2}\in
\mathbb{Z}
^{+}.$
\end{definition}



\begin{lemma}
Let $R$ be a $G$-graded ring, $M$
a graded $R$-module, $I=\oplus _{g\in G}I_{g}$ a graded ideal of $R$ and $N$
a graded $G2$-absorbing submodule of $M$. If $r_{g}\in h(R)$, $m_{\lambda
}\in h(M)$ and $h\in G$ such that $r_{g}I_{h}m_{\lambda }\subseteq N$, then
either $r_{g}\in Gr((N:_{R}m_{\lambda }))$ or $I_{h}\subseteq
Gr((N:_{R}m_{\lambda }))$ or $r_{g}I_{h}\subseteq (N:_{R}M)$.
\end{lemma}

\begin{proof}
 Suppose that $r_{g}\not\in Gr((N:_{R}m_{\lambda }))$ and $%
r_{g}I_{h}\not\subseteq (N:_{R}M)$. Hence, there exists $i_{h}^{\prime }\in
I_{h}$ such that $r_{g}i_{h}^{\prime }\not\in (N:_{R}M)$. Now, $%
r_{g}i_{h}^{\prime }m_{\lambda }\in N$ implies that $i_{h}^{\prime }\in
Gr((N:_{R}m_{\lambda }))$ as $N$ is a graded $G2$-absorbing submodule of $M$%
. Now, Let $i_{h}\in I_{h}$, then $r_{g}(i_{h}+i_{h}^{\prime })m_{\lambda
}\in N$. Hence, either $i_{h}+i_{h}^{\prime }\in Gr((N:_{R}m_{\lambda }))$
or $r_{g}(i_{h}+i_{h}^{\prime })\in (N:_{R}M)$. If $i_{h}+i_{h}^{\prime }\in
Gr((N:_{R}m_{\lambda }))$, then $i_{h}\in Gr((N:_{R}m_{\lambda }))$ since $%
i_{h}^{\prime }\in Gr((N:_{R}m_{\lambda }))$. If $r_{g}(i_{h}+i_{h}^{\prime
})\in (N:_{R}M)$, then $r_{g}i_{h}\not\in (N:_{R}M)$. But $%
r_{g}i_{h}m_{\lambda }\in N$ which yields that $i_{h}\in
Gr((N:_{R}m_{\lambda }))$ as $N$ is a graded $G2$-absorbing submodule of $M$%
. Therefore, we get $I_{h}\subseteq Gr((N:_{R}m_{\lambda }))$.
\end{proof}
\begin{lemma}
 Let $R$ be a $G$-graded ring, $M$ a
graded $R$-module, $I=\oplus _{g\in G}I_{g},J=\oplus _{g\in G}J_{g}$ be two
graded ideals of $R$ and $N$ a graded $G2$-absorbing submodule of $M$. If $%
m_{\lambda }\in h(M)$ and $g,h\in G$ such that $J_{g}I_{h}m_{\lambda
}\subseteq N$, then either $J_{g}\subseteq Gr((N:_{R}m_{\lambda }))$ or $%
I_{h}\subseteq Gr((N:_{R}m_{\lambda }))$ or $J_{g}I_{h}\subseteq (N:_{R}M)$.
\end{lemma}
\begin{proof}
 Suppose that $J_{g}\not\subseteq Gr((N:_{R}m_{\lambda }))$ and $%
I_{h}\not\subseteq Gr((N:_{R}m_{\lambda }))$, then there exist $%
j_{g}^{\prime }\in J_{g}$ and $i_{h}^{\prime }\in I_{h}$ such that $%
j_{g}^{\prime }\not\in Gr((N:_{R}m_{\lambda }))$ and $i_{h}^{\prime }\not\in
Gr((N:_{R}m_{\lambda })).$ Now, by \textbf{Lemma 2.3}, $j_{g}^{\prime
}I_{h}m_{\lambda }\subseteq N$ yields that $j_{g}^{\prime }I_{h}\subseteq
(N:_{R}M)$ and so $(J_{g}\backslash Gr((N:_{R}m_{\lambda })))I_{h}\subseteq
(N:_{R}M)$. Also, since $i_{h}^{\prime }\not\in Gr((N:_{R}m_{\lambda }))$, $%
J_{g}(I_{h}\backslash Gr((N:_{R}m_{\lambda })))\subseteq (N:_{R}M)$. Hence,
we get $j_{g}^{\prime }i_{h}^{\prime }\in (N:_{R}M)$. Now, let $j_{g}\in
J_{g}$ and $i_{h}\in I_{h},$ then $j_{g}^{\prime }i_{h}\in (N:_{R}M)$ and $%
j_{g}i_{h}^{\prime }\in (N:_{R}M)$. Thus, $(j_{g}+j_{g}^{\prime
})(i_{h}+i_{h}^{\prime })m_{\lambda }\in N$ yields that either $%
j_{g}+j_{g}^{\prime }\in Gr((N:_{R}m_{\lambda }))$ or $i_{h}+i_{h}^{\prime
}\in Gr((N:_{R}m_{\lambda }))$ or $(j_{g}+j_{g}^{\prime
})(i_{h}+i_{h}^{\prime })\in (N:_{R}M)$. If $j_{g}+j_{g}^{\prime }\in
Gr((N:_{R}m_{\lambda }))$, then $j_{g}\not\in Gr((N:_{R}m_{\lambda }))$
which follows that $j_{g}\in J_{g}\backslash Gr((N:_{R}m_{\lambda }))$ and
then $j_{g}i_{h}\in (N:_{R}M)$. Similarly, if $i_{h}+i_{h}^{\prime }\in
Gr((N:_{R}m_{\lambda }))$, then $j_{g}i_{h}\in (N:_{R}M)$. Now, if $%
(j_{g}+j_{g}^{\prime })(i_{h}+i_{h}^{\prime })=j_{g}i_{h}+j_{g}i_{h}^{\prime
}+j_{g}^{\prime }i_{h}+j_{g}^{\prime }i_{h}^{\prime }\in (N:_{R}M)$, then we
get $j_{g}i_{h}\in (N:_{R}M)$. Therefore, $J_{g}I_{h}\subseteq (N:_{R}M)$.
\end{proof}


The following Theorem gives us a characterization of a graded $G2$-absorbing submodule.

\begin{theorem}
Let $R$ be a $G$-graded ring, $M$ a
graded $R$-module and $N$ a proper graded submodule of $M$. Let $I=\oplus
_{g\in G}I_{g},J=\oplus _{g\in G}J_{g}$ be two graded ideals of $R$ and $%
K=\oplus _{g\in G}K_{g}$ a graded submodule of $M.$ Then the following
statements are equivalent:
 \begin{enumerate}[\upshape (i)]
   \item $N$ is a graded $G2$-absorbing submodule of $M$.

   \item If $g,h,\lambda \in G$ with $J_{g}I_{h}K_{\lambda }\subseteq N$, then
either $J_{g}\subseteq Gr((N:_{R}K_{\lambda }))$ or $I_{h}\subseteq
Gr((N:_{R}K_{\lambda }))$ or $J_{g}I_{h}\subseteq (N:_{R}M)$.

\end{enumerate}

\end{theorem}
\begin{proof}
 $(i)\Rightarrow (ii)$ Suppose that $N$ is a graded $G2$-absorbing
submodule of $M$ and let $g,h,\lambda \in G$ with $J_{g}I_{h}K_{\lambda
}\subseteq N$ and $J_{g}I_{h}\not\subseteq (N:_{R}M)$. Now, let $k_{\lambda
}\in K_{\lambda },$ then since $J_{g}I_{h}k_{\lambda }\subseteq N,$ either $%
J_{g}\subseteq Gr((N:_{R}k_{\lambda }))$ or $I_{h}\subseteq
Gr((N:_{R}k_{\lambda }))$ by \textbf{Lemma 2.4}. Assume that there exist $%
k_{1_{\lambda }},k_{2_{\lambda }}\in K_{\lambda }$ such that $%
J_{g}\not\subseteq Gr((N:_{R}k_{1_{\lambda }}))$ and $I_{h}\not\subseteq
Gr((N:_{R}k_{2_{\lambda }}))$. Hence, $I_{h}\subseteq
Gr((N:_{R}k_{1_{\lambda }}))$ and $J_{g}\subseteq Gr((N:_{R}k_{2_{\lambda
}}))$. Now, since $J_{g}I_{h}(k_{1_{\lambda }}+k_{2_{\lambda }})\subseteq N,$
either $J_{g}\subseteq Gr((N:_{R}k_{1_{\lambda }}+k_{2_{\lambda }}))$ or $%
I_{h}\subseteq Gr((N:_{R}k_{1_{\lambda }}+k_{2_{\lambda }}))$ which yields
that either $J_{g}\subseteq Gr((N:_{R}k_{1_{\lambda }}))$ and $%
I_{h}\subseteq Gr((N:_{R}k_{2_{\lambda }}))$, a contradiction. Therefore,
either $J_{g}\subseteq Gr((N:_{R}K_{\lambda }))$ or $I_{h}\subseteq
Gr((N:_{R}K_{\lambda }))$.

$(ii)\Rightarrow (i)$ Assume that $(ii)$ holds.\ Let $r_{g},s_{h}\in h(R)$
and $m_{\lambda }\in h(M)$ such that $r_{g}s_{h}m_{\lambda }\in N.$ Now, let
$J=\langle r_{g}\rangle $ and $I=\langle s_{h}\rangle $ be two graded ideals
of $R$ generated by $r_{g}$ and $s_{h},$ respectively, and let $K=\langle
m_{\lambda }\rangle $ be a graded submodule of $M$ generated by $m_{\lambda
}.$ Thus, $J_{g}I_{h}K_{\lambda }\subseteq N,$ then we get either $%
J_{g}\subseteq Gr((N:_{R}K_{\lambda }))$ or $I_{h}\subseteq
Gr((N:_{R}K_{\lambda }))$ or $J_{g}I_{h}\subseteq (N:_{R}M).$ Hence, either $%
r_{g}\subseteq Gr((N:_{R}m_{\lambda }))$ or $s_{h}\subseteq
Gr((N:_{R}m_{\lambda }))$ or $r_{g}s_{h}\subseteq (N:_{R}M).$ Therefore, $N$
is a graded $G2$-absorbing submodule of $M$.
 \end{proof}

Recall from \cite{14} that if $I=\oplus _{g\in G}I_{g}$
is a graded ideal of a $G$-graded ring $R,$ then $I_{e}$ is said to be an $e$%
-$2$-absorbing primary ideal of $R_{e}$ if $I_{e}\not=R_{e};$ and whenever $%
r_{e},s_{e},t_{e}\in R_{e}$ with $r_{e}s_{e}t_{e}\in I_{e}$, then either $%
r_{e}s_{e}\in I_{e}$ or $r_{e}t_{e}\in Gr(I_{e})$ or $s_{e}t_{e}\in Gr(I_{e})
$.

\begin{theorem}
Let $R$ be a $G$-graded ring, $M$ a
graded $R$-module, $N=\oplus _{g\in G}N_{g}$, $K=\oplus _{g\in G}K_{g}$ be
two graded submodules of $M$ and $g\in G.$ If $N_{g}$ is a $g$-$G2$%
-absorbing submodule of an $R_{e}$-module $M_{g},$ then the following
statements hold:
\begin{enumerate}[\upshape (i)]
   \item If $K_{g}\not\subseteq N_{g}$, then $(N_{g}:_{R_{e}}K_{g})$ is an $e$-$%
2$-absorbing primary ideal of $R_{e}$.
    \item $(N_{g}:_{R_{e}}M_{g})$ is an $e$-$2$-absorbing primary ideal of $%
R_{e}$.

   \end{enumerate}
\end{theorem}
\begin{proof}
  $(i)$ Let $r_{e},s_{e},t_{e}\in R_{e}$ such that $%
r_{e}s_{e}t_{e}\in (N_{g}:_{R_{e}}K_{g})$. Then either $%
r_{e}^{n_{1}}t_{e}K_{g}\subseteq N_{g}$ for some $n_{1}\in
\mathbb{Z}
^{+}$ or $s_{e}^{n_{2}}t_{e}K_{g}\subseteq N_{g}$ for some $n_{2}\in
\mathbb{Z}
^{+}$ or $r_{e}s_{e}M_{g}\subseteq N_{g}$ as $N_{g}$ is a $g$-$G2$-absorbing
submodule of $M_{g}$. Therefore, $(r_{e}t_{e})^{n_{1}}K_{g}\subseteq N_{g}$
or $(s_{e}t_{e})^{n_{2}}K_{g}\subseteq N_{g}$ or $r_{e}s_{e}K_{g}\subseteq
N_{g}$. Therefore, $(N_{g}:_{R_{e}}K_{g})$ is an $e$-$2$-absorbing primary
ideal of $R_{e}$.

$(ii)$ The proof comes from $M_{g}\not\subseteq N_{g}$ and part $(i)$.
\end{proof}
Recall from \cite{5} that if $I=\oplus _{g\in G}I_{g}$ is a
graded ideal of a $G$-graded ring $R,$ then $I_{e}$ is said to be an $e$-$2$%
-absorbing ideal of $R_{e}$ if $I_{e}\not=R_{e};$ and whenever $%
r_{e},s_{e},t_{e}\in R_{e}$ with $r_{e}s_{e}t_{e}\in I_{e}$, then either $%
r_{e}s_{e}\in I_{e}$ or $r_{e}t_{e}\in I_{e}$ or $s_{e}t_{e}\in I_{e}$.

\begin{theorem}
 Let $R$ be a $G$-graded ring and $%
I=\oplus _{g\in G}I_{g}$ a graded ideal of $R.$ If $I_{e}$ is an $e$-$2$%
-absorbing primary ideal of $R_{e}$, then $Gr(I_{e})$ is an $e$-$2$%
-absorbing ideal of $R_{e}$.
\end{theorem}
\begin{proof}
 Let $r_{e},s_{e},t_{e}\in R_{e}$ such that $r_{e}s_{e}t_{e}\in
Gr(I_{e})$ and neither $r_{e}t_{e}\in Gr(I_{e})$ nor $s_{e}t_{e}\in
Gr(I_{e}) $. Hence, there exists $n\in
\mathbb{Z}
^{+}$ such that $(r_{e}s_{e}t_{e})^{n}=r_{e}^{n}s_{e}^{n}t_{e}^{n}\in I_{e}$%
. Since $I_{e}$ is an $e$-$2$-absorbing primary and $r_{e}t_{e}\not\in
Gr(I_{e})$ and $s_{e}t_{e}\not\in Gr(I_{e})$, we get $%
r_{e}^{n}s_{e}^{n}=(r_{e}s_{e})^{n}\in I_{e}$. Thus, $r_{e}s_{e}\in
Gr(I_{e}) $. Therefore, $Gr(I_{e})$ is an $e$-$2$-absorbing ideal of $R_{e}$.
\end{proof}
\begin{corollary}
 Let $R$ be a $G$-graded ring,
$M$ a graded $R$-module, $N=\oplus _{g\in G}N_{g}$ a graded submodule of $M$
and $g\in G.$ If $N_{g}$ is a $g$-$G2$-absorbing submodule of an $R_{e}$%
-module $M_{g}$, then $Gr((N_{g}:_{R_{e}}M_{g}))$ is an $e$-$2$-absorbing
ideal of $R_{e}$.
\end{corollary}
\begin{proof}
 We have $(N_{g}:_{R_{e}}M_{g})$ is an $e$-$2$-absorbing primary
ideal of $R_{e}$ by \textbf{Theorem 2.6}. Therefore, $%
Gr((N_{g}:_{R_{e}}M_{g}))$ is an $e$-$2$-absorbing ideal of $R_{e}$ by
\textbf{Theorem 2.7}.
\end{proof}

Let $R$ be a $G$-graded ring, $M$ a graded $R$-module, $N=\oplus _{g\in
G}N_{g}$ a graded submodule of $M$ and $g\in G.$ An $R_{e}$-module $M_{g}$
is said to be a $g$-multiplication module if for every submodule $N_{g}$ of
\thinspace $M_{g}$ there exists an ideal $I_{e}$ of $R_{e}$ such that $%
N_{g}=I_{e}M_{g}$, (see \cite{21}).
\begin{corollary}
 Let $R$ be a $G$-graded ring, $M$ a
graded $R$-module, $N=\oplus _{g\in G}N_{g}$ a graded submodule of $M$ and $%
g\in G.$ If $M_{g}$ is a $g$-multiplication $R_{e}$-module and $N_{g}$ is a $%
g$-$G2$-absorbing submodule of $M_{g}$ such that $%
Gr((N_{g}:_{R_{e}}M_{g}))=(N_{g}:_{R_{e}}M_{g})$, then $N_{g}$ is a $g$-$2$%
-absorbing submodule of $M_{g}$.
\end{corollary}
\begin{proof}
 By \textbf{Theorem 2.6} we have $(N_{g}:_{R_{e}}M_{g})$ is an $e$-$2
$-absorbing primary ideal of $R_{e}$. Hence, $%
(N_{g}:_{R_{e}}M_{g})=Gr((N_{g}:_{R_{e}}M_{g}))$ is an $e$-$2$-absorbing
ideal of $R_{e}$ by \textbf{Corollary 2.8}. Now, let $r_{e},s_{e}\in R_{e}$
and $m_{g}\in M_{g}$ such that $r_{e}s_{e}m_{g}\in N_{g}$. Hence, $%
r_{e}s_{e}R_{e}m_{g}\subseteq N_{g}$ and as $M_{g}$ is a $g$-multiplication
module we get $r_{e}s_{e}(R_{e}m_{g}:_{R_{e}}M_{g})M_{g}\subseteq N_{g}$.
Thus, $r_{e}s_{e}(R_{e}m_{g}:_{R_{e}}M_{g})\subseteq (N_{g}:_{R_{e}}M_{g})$,
which yields that either $r_{e}s_{e}\in (N_{g}:_{R_{e}}M_{g})$ or $%
r_{e}(R_{e}m_{g}:_{R_{e}}M_{g})\subseteq (N_{g}:_{R_{e}}M_{g})$ or $%
s_{e}(R_{e}m_{g}:_{R_{e}}M_{g})\subseteq (N_{g}:_{R_{e}}M_{g})$. So, either $%
r_{e}s_{e}\in (N_{g}:_{R_{e}}M_{g})$ or $%
r_{e}R_{e}m_{g}=r_{e}(R_{e}m_{g}:_{R_{e}}M_{g})M_{g}\subseteq N_{g}$ or $%
s_{e}R_{e}m_{g}=s_{e}(R_{e}m_{g}:_{R_{e}}M_{g})M_{g}\subseteq N_{g}$, then
either $r_{e}s_{e}\in (N_{g}:_{R_{e}}M_{g})$ or $r_{e}m_{g}\in N_{g}$ or $%
s_{e}m_{g}\in N_{g}$. Therefore,\ $N_{g}$ is a $g$-$2$-absorbing submodule
of $M_{g}$.
\end{proof}

Recall that from \cite{10} that a graded zero-divisor on a graded $R$-module $M$ is an element $r_{g}\in h(R)$ for which there exists $m_{h}\in
h(M)$ such that $m_{h}\not=0$ but $r_{g}m_{h}=0$. The set of all graded zero-divisors on $M$ is denoted by $G$-$Zdv_{R}(M)$.

The following result studies the behavior of graded $G2$-absorbing
submodules under localization.

\begin{theorem}
 Let $R$ be a $G$-graded ring, $M$ a
graded $R$-module and $S\subseteq h(R)$ be a multiplicatively closed subset
of $R.$
\begin{enumerate}[\upshape (i)]
   \item  If $N$ is a graded $G2$-absorbing submodule of $M$ such that $%
(N:_{R}M)\cap S=\emptyset $, then $S^{-1}N$ is a graded $G2$-absorbing
submodule of $S^{-1}M$.

    \item If $S^{-1}N$ is a graded $G2$-absorbing submodule of $S^{-1}M$ such
that $G$-$Zdv_{R}(M/N)\cap S=\emptyset $, then $N$ is a graded $G2$%
-absorbing submodule of $M$.
   \end{enumerate}

\end{theorem}
\begin{proof}
$(i)$ Since $(N:_{R}M)\cap S=\emptyset ,$\ $S^{-1}N$ is a proper
graded submodule of $S^{-1}M.$ Now, let $\frac{r_{g}}{s_{1}},\frac{s_{h}}{%
s_{2}}\in h(S^{-1}R)$ and $\frac{m_{\lambda }}{s_{3}}\in h(S^{-1}M)$ such
that $\frac{r_{g}}{s_{1}}\frac{s_{h}}{s_{2}}\frac{m_{\lambda }}{s_{3}}=\frac{%
r_{g}s_{h}m_{\lambda }}{s_{1}s_{2}s_{3}}\in S^{-1}N.$ Hence, there exists $%
s_{4}\in S$ such that $s_{4}r_{g}s_{h}m_{\lambda }\in N$ which yields that
either $r_{g}^{k_{1}}s_{4}m_{\lambda }\in N$ for some $k_{1}\in
\mathbb{Z}
^{+}$ or $s_{h}^{k_{2}}s_{4}m_{\lambda }\in N$ for some $k_{2}\in
\mathbb{Z}
^{+}$ or $r_{g}s_{h}\in (N:_{R}M)$ as $N$ is a graded $G2$-absorbing
submodule of $M$. Thus, either $(\frac{r_{g}}{s_{1}})^{k_{1}}(\frac{%
m_{\lambda }}{s_{3}})=(\frac{r_{g}^{k_{1}}s_{4}m_{\lambda }}{%
s_{1}^{k_{1}}s_{4}s_{3}})\in S^{-1}N$ or $(\frac{s_{h}}{s_{2}})^{k_{2}}(%
\frac{m_{\lambda }}{s_{3}})=(\frac{s_{h}^{k_{2}}s_{4}m_{\lambda }}{%
s_{2}^{k_{2}}s_{4}s_{3}})\in S^{-1}N$ or $\frac{r_{g}}{s_{1}}\frac{s_{h}}{%
s_{2}}=\frac{r_{g}s_{h}}{s_{1}s_{2}}\in S^{-1}(N:_{R}M)\subseteq
(S^{-1}N:_{S^{-1}R}S^{-1}M)$. Therefore, $S^{-1}N$ is a graded $G2$%
-absorbing submodule of $S^{-1}M$.

$(ii)$ Let $r_{g},s_{h}\in h(R)$ and $m_{\lambda }\in h(M)$ such that $%
r_{g}s_{h}m_{\lambda }\in N$, so$\frac{r_{g}}{1}\frac{s_{h}}{1}\frac{%
m_{\lambda }}{1}=\frac{r_{g}s_{h}m_{\lambda }}{1}\in S^{-1}N$. Hence, either
$(\frac{r_{g}}{1})^{k_{1}}(\frac{m_{\lambda }}{1})\in S^{-1}N$ for some $%
k_{1}\in
\mathbb{Z}
^{+}$ or $(\frac{s_{h}}{1})^{k_{2}}(\frac{m_{\lambda }}{1})\in S^{-1}N$ for
some $k_{2}\in
\mathbb{Z}
^{+}$ or $\frac{r_{g}s_{h}}{1}\in (S^{-1}N:_{S^{-1}R}S^{-1}M)$ as $S^{-1}N$
is a graded $G2$-absorbing submodule of $S^{-1}M$. If $\frac{r_{g}s_{h}}{1}%
\in (S^{-1}N:_{S^{-1}R}S^{-1}M)=S^{-1}(N:_{R}M)$, then $r_{g}s_{h}\in
(N:_{R}M)$ and we get the result. Otherwise, if $\frac{r_{g}^{k_{1}}m_{%
\lambda }}{1}\in S^{-1}N$ there exists $t\in S$ such that $%
tr_{g}^{k_{1}}m_{\lambda }\in N$ which yields that $r_{g}^{k_{1}}m_{\lambda
}\in N$ since $G$-$Zdv_{R}(M/N)\cap S=\emptyset $. Similarly, if $\frac{%
s_{h}^{k_{2}}m_{\lambda }}{1}\in S^{-1}N,$ we get $s_{h}^{k_{2}}m_{\lambda
}\in N.$ Therefore, $N$ is a graded $G2$-absorbing submodule of $M$.
\end{proof}
%
 \section{Graded $G2$-absorbing submodules over Gr-Noetherian ring}

Let $R$ be a $G$-graded ring, $M$ a graded $R$-module, $N=\oplus _{g\in
G}N_{g}$ a graded submodule of $M$ and $g\in G.$ We say that $N_{g}$ is a $g$%
-idempotent submodule of an $R_{e}$-module $M_{g}$ if $%
N_{g}=(N_{g}:_{R_{e}}M_{g})^{2}M_{g}.$ Also, $M_{g}$ is called a fully $g$%
-idempotent if every submodule is a $g$-idempotent. It is easy to see that
every fully $g$-idempotent module is a $g$-multiplication.

 A $G$-graded ring $R$ is called $gr$-Noetherian if it satisfies the
ascending chain condition on graded ideals of $R$. Equivalently, $R$ is
$gr$-Noetherian if and only if every graded ideal of $R$ is finitely generated (see \cite{27}).

\begin{lemma}
Let $R$ be a $G$-graded ring, $I=\oplus _{g\in G}I_{g}$
and $J=\oplus _{g\in G}J_{g}$ be two graded ideals of $R.$ If $I_{e}$ is an $%
e$-$2$-absorbing primary ideal of $R_{e}\ $and $r_{e},s_{e}\in R_{e}$ such
that $r_{e}s_{e}J_{e}\subseteq I_{e}$, then either $r_{e}s_{e}\in I_{e}$ or $%
r_{e}J_{e}\subseteq Gr(I_{e})$ or $s_{e}J_{e}\subseteq Gr(I_{e})$.

\end{lemma}
\begin{proof}
 Assume that neither $r_{e}s_{e}\in I_{e}$ nor $r_{e}J_{e}\subseteq
Gr(I_{e})$ nor $s_{e}J_{e}\subseteq Gr(I_{e})$. Then there exist $%
j_{1_{e}},j_{2_{e}}\in J_{e}$ with $r_{e}j_{1_{e}}\not\in Gr(I_{e})$ and $%
s_{e}j_{2_{e}}\not\in Gr(I_{e})$. Now, since $r_{e}s_{e}j_{1_{e}}\in I_{e}$
and $r_{e}s_{e}j_{2_{e}}\in I_{e},$ $s_{e}j_{1_{e}}\in Gr(I_{e})$ and $%
r_{e}j_{2_{e}}\in Gr(I_{e})$ as $I_{e}$ is an $e$-$2$-absorbing primary
ideal of $R_{e}$. Now, $r_{e}s_{e}(j_{1_{e}}+j_{2_{e}})\in I_{e}$ and $%
r_{e}s_{e}\not\in I_{e}$, we have either $r_{e}(j_{1_{e}}+j_{2_{e}})\in
Gr(I_{e})$ or $s_{e}(j_{1_{e}}+j_{2_{e}})\in Gr(I_{e})$ and then either $%
r_{e}j_{1_{e}}\in Gr(I_{e})$ or $s_{e}j_{2_{e}}\in Gr(I_{e}),$ a
contradiction. Therefore, either $r_{e}s_{e}\in I_{e}$ or $%
r_{e}J_{e}\subseteq Gr(I_{e})$ or $s_{e}J_{e}\subseteq Gr(I_{e})$.
\end{proof}

\begin{theorem}
Let $R$ be a $G$-graded $gr$-Noetherian ring, $M$ a
graded $R$-module, $N=\oplus _{g\in G}N_{g}$ a graded submodule of $M$ and $%
g\in G$. If $M_{g}$ is a fully $g$-idempotent $R_{e}$-module and $%
(N_{g}:_{R_{e}}M_{g})$ is an $e$-$2$-absorbing primary ideal of $R_{e}$,
then $N_{g}$ is a $g$-$G2$-absorbing submodule of an $R_{e}$-module $M_{g}$.
\end{theorem}
\begin{proof}
Let $r_{e},s_{e}\in R_{e}$ and $K=\oplus _{h\in G}K_{h}$ be a
graded submodule of $M$\ such that $r_{e}s_{e}K_{g}\subseteq N_{g}$. Now,
since $M_{g}$ is a fully $g$-idempotent, $M_{g}$ is a $g$-multiplication
module, so we get $r_{e}s_{e}(K_{g}:_{R_{e}}M_{g})M_{g}\subseteq N_{g}$ and
then $r_{e}s_{e}(K_{g}:_{R_{e}}M_{g})\subseteq (N_{g}:_{R_{e}}M_{g})$.
Hence, \textbf{by Lemma 3.1,} either $r_{e}(K_{g}:_{R_{e}}M_{g})\subseteq
Gr((N_{g}:_{R_{e}}M_{g}))$ or $s_{e}(K_{g}:_{R_{e}}M_{g})\subseteq
Gr((N_{g}:_{R_{e}}M_{g}))$ or $r_{e}s_{e}\in (N_{g}:_{R_{e}}M_{g})$ as $%
(N_{g}:_{R_{e}}M_{g})$ is an $e$-$2$-absorbing primary ideal of $R_{e}$. If $%
r_{e}s_{e}\in (N_{g}:_{R_{e}}M_{g})$, then we get the result. Now, since $R$
is $gr$-Noetherian, then so $R_{e}.$ Hence, if $r_{e}(K_{g}:_{R_{e}}M_{g})%
\subseteq Gr((N_{g}:_{R_{e}}M_{g}))$, then $%
(r_{e}(K_{g}:_{R_{e}}M_{g}))^{n_{1}}\subseteq (N_{g}:_{R_{e}}M_{g}),$ for
some $n_{1}\in
\mathbb{Z}
^{+}$ which follows that $%
r_{e}^{n_{1}}K_{g}=r_{e}^{n_{1}}(K_{g}:_{R_{e}}M_{g})^{n_{1}}M_{g}\subseteq
(N_{g}:_{R_{e}}M_{g})M_{g}=N_{g}$ as $M_{g}$ is a fully $g$-idempotent.
Similarly, if $s_{e}(K_{g}:_{R_{e}}M_{g})\subseteq Gr((N_{g}:_{R_{e}}M_{g})),
$ then $s_{e}^{n_{2}}K_{g}\subseteq N_{g}$, for some $n_{2}\in
\mathbb{Z}
^{+}.$ Therefore, $N_{g}$ is a $g$-$G2$-absorbing submodule of $M_{g}$.
\end{proof}
The following example shows that \textbf{Theorem 3.2} is not true in
general.
\begin{example}
Let $G=%
\mathbb{Z}
_{2},$ then $R=%
\mathbb{Z}
$ is a $G$-graded ring with $R_{0}=%
\mathbb{Z}
$ and $R_{1}=\{0\}.$ Let $M=%
\mathbb{Q}
$ be a graded $R$-module with $M_{0}=%
\mathbb{Q}
$ and $M_{1}=\{0\}$ where $M_{0}$ is not a fully $0$-idempotent. Now,
consider the graded submodule $N=%
\mathbb{Z}
$ of $M.$ Then $N_{0}$ is not a $0$-$G2$-absorbing submodule of $M_{0}$
since $2\cdot 3\cdot \frac{1}{6}\in
\mathbb{Z}
$ and neither $2\cdot 3\in (%
\mathbb{Z}
:_{%
\mathbb{Z}
}%
\mathbb{Q}
)=\{0\}$ nor $2\in Gr((%
\mathbb{Z}
:_{%
\mathbb{Z}
}\frac{1}{6}))$ nor $3\in Gr((%
\mathbb{Z}
:_{%
\mathbb{Z}
}\frac{1}{6})).$ However, easy computations show that $(%
\mathbb{Z}
:_{%
\mathbb{Z}
}%
\mathbb{Q}
)=\{0\}$ is a $0$-$2$-absorbing primary ideal of $%
\mathbb{Z}
.$
\end{example}


\begin{lemma}
Let $R_{i}$ be a $G$-graded ring, $M_{i}$ a
graded $R_{i}$-module, for $i=1,2$ and $g\in G.$ Let $R=R_{1}\times R_{2}$
and $M=M_{1}\times M_{2}$. Then $M_{i_{g}}$ is a fully $g$-idempotent $%
R_{i_{e}}$-module, for $i=1,2$ if and only if $M_{g}$ is a fully $g$%
-idempotent $R_{e}$-module.
\end{lemma}
\begin{proof}
Suppose that $M_{g}$ is a fully $g$-idempotent $R_{e}$-module and $%
N_{1_{g}}$ is a submodule of an $R_{1_{e}}$-module $M_{1_{g}}$. Then $%
N_{g}=N_{1_{g}}\times \{0\}_{2_{g}}$ is a submodule of $M_{g}$. Hence, $%
N_{g}=(N_{g}:_{R_{e}}M_{g})^{2}M_{g}=(N_{1_{g}}:_{R_{1_{e}}}M_{1_{g}})^{2}M_{1_{g}}\times (\{0\}_{2_{g}}:_{R_{2e}}M_{2_{g}})^{2}M_{2_{g}}
$. Thus, $N_{1_{g}}=(N_{1_{g}}:_{R_{1_{e}}}M_{1_{g}})^{2}M_{1_{g}}$.
Therefore, $M_{1_{g}}$ is a fully $g$-idempotent $R_{1_{e}}$-module.
Similarly, $M_{2_{g}}$ is a fully $g$-idempotent $R_{2_{e}}$-module.
Conversely, let $N_{g}$ be a submodule of $M_{g}$. Then $N_{g}=N_{1_{g}}%
\times N_{2_{g}}$ for some submodules $N_{1_{g}}$ of $M_{1_{g}}$ and $%
N_{2_{g}}$ of $M_{2_{g}}$. But $%
N_{i_{g}}=(N_{i_{g}}:_{R_{i_{e}}}M_{i_{g}})^{2}M_{i_{g}},$ for $i=1,2$, so $%
N_{g}=(N_{1_{g}}:_{R_{1_{e}}}M_{1_{g}})^{2}M_{1_{g}}\times
(N_{2_{g}}:_{R_{2_{e}}}M_{2_{g}})^{2}M_{2_{g}}=(N_{g}:_{R_{e}}M_{g})^{2}M_{g}
$. Therefore, $M_{g}$ is a fully $g$-idempotent $R_{e}$-module.
\end{proof}
Let $R$ be a $G$-graded ring and $I=\oplus _{g\in G}I_{g}$  be a graded
ideal of $R.$ Then $I_{e}$ is said to be an $e$-prime ideal of $R_{e}$ if
whenever $r_{e},s_{e}\in R_{e}$ with $r_{e}s_{e}\in I_{e},$ implies either $%
r_{e}\in I_{e}$ or $s_{e}\in I_{e}.$ Also, $I_{e}$ is said to be an $e$%
-primary ideal of $R_{e}$ if whenever $r_{e},s_{e}\in R_{e}$ with $%
r_{e}s_{e}\in I_{e},$ implies either $r_{e}\in I_{e}$ or $s_{e}\in Gr(I_{e})$
(see \cite{30}).

\begin{lemma}
Let $R$ be a $G$-graded ring and $I_{i}=\oplus _{g\in
G}I_{i_{g}}$ is a graded ideal of $R$, for $i=1,2$. If $I_{i_{e}}$ is an $e$-%
$P_{i_{e}}$-primary ideal of $R_{e}$ for some $e$-prime ideal $P_{i_{e}}$ of
$R_{e}$, for $i=1,2,$ then $I_{1_{e}}\cap I_{2_{e}}$ is an $e$-$2$-absorbing
primary ideal of $R_{e}$.
\end{lemma}
\begin{proof}
 Let $J_{e}=I_{1_{e}}\cap I_{2_{e}}$, then $Gr(J_{e})=P_{1_{e}}\cap
P_{2_{e}}$ is an $e$-$2$-absorbing ideal of $R_{e}$. Now, let $%
r_{e},s_{e},t_{e}\in R_{e}$ such that $r_{e}s_{e}t_{e}\in J_{e}$ and neither
$r_{e}t_{e}\in Gr(J_{e})$ nor $s_{e}t_{e}\in Gr(J_{e})$. Thus, $%
r_{e},s_{e},t_{e}\not\in Gr(J_{e})=P_{1_{e}}\cap P_{2_{e}}$. Since $%
Gr(J_{e})=P_{1_{e}}\cap P_{2_{e}}$ is an $e$-$2$-absorbing ideal of $R_{e}$
and $r_{e}t_{e},s_{e}t_{e}\not\in Gr(J_{e})$, $r_{e}s_{e}\in Gr(J_{e})$.
Now, suppose that $r_{e}\in P_{1_{e}}$, then $r_{e}\not\in Gr(J_{e})$ and $%
r_{e}s_{e}\in Gr(J_{e})\subseteq P_{2_{e}}$ implies $r_{e}\not\in P_{2_{e}}$
and $s_{e}\in P_{2_{e}}$. Thus, $s_{e}\not\in P_{1_{e}}$. If $r_{e}\in
I_{1_{e}}$ and $s_{e}\in I_{2_{e}}$, then $r_{e}s_{e}\in J_{e}$ and we are
done. Now, Suppose that $r_{e}\not\in I_{1_{e}}$. Since $I_{1_{e}}$ is an $e$%
-$P_{1_{e}}$-primary ideal of $R_{e}$ and $r_{e}\not\in I_{1_{e}}$, $%
s_{e}t_{e}\in P_{1_{e}}$. Also, since $s_{e}\in P_{2_{e}}$ and $%
s_{e}t_{e}\in P_{1_{e}}$, $s_{e}t_{e}\in Gr(J_{e})$, a contradiction. So $%
r_{e}\in I_{1_{e}}$. Similarly, suppose that $s_{e}\not\in I_{2_{e}}$. Since
$I_{2_{e}}$ is an $e$-$P_{2_{e}}$-primary ideal of $R_{e}$ and $s_{e}\not\in
I_{2_{e}}$, $r_{e}t_{e}\in P_{2_{e}}$. Also, since $r_{e}t_{e}\in P_{2_{e}}$
and $r_{e}\in P_{1_{e}}$, $r_{e}t_{e}\in Gr(J_{e})$, a contradiction. So, $%
s_{e}\in I_{2_{e}}$ and then $r_{e}s_{e}\in J_{e}$. Therefore, $%
I_{1_{e}}\cap I_{2_{e}}$ is an $e$-$2$-absorbing primary ideal of $R_{e}$
\end{proof}

\begin{theorem}
Let $R_{1}$, $R_{2}$ be two $G$-graded rings such
that $R=R_{1}\times R_{2}$ and $J=\oplus _{g\in G}J_{g}$ be a proper graded
ideal of $R$. Then the following statements are equivalent.
\begin{enumerate}[\upshape (i)]
   \item $J_{e}$ is an $e$-$2$-absorbing primary ideal of $R_{e}$.

    \item Either $J_{e}=I_{1_{e}}\times \ R_{2_{e}}$ for some $e$-$2$-absorbing
primary ideal $I_{1_{e}}$ of $R_{1_{e}}$ or $J_{e}=R_{1_{e}}\times I_{2_{e}}$
for some $e$-$2$-absorbing primary ideal $I_{2_{e}}$ of $R_{2_{e}}$ or $%
J_{e}=I_{1_{e}}\times I_{2_{e}}$ for some $e$-primary ideal $I_{1_{e}}$ of $%
R_{1_{e}}$ and some $e$-primary ideal $I_{2_{e}}$ of $R_{2_{e}}$.
   \end{enumerate}
  \end{theorem}
\begin{proof}
$(i)\Rightarrow (ii)$ Suppose that $J_{e}$ is an $e$-$2$-absorbing primary
ideal of $R_{e}$. So, $J_{e}=I_{1_{e}}\times I_{2_{e}}$ where $I_{1_{e}}$
and $I_{2_{e}}$ are two ideals of $R_{1_{e}}$ and $R_{2_{e}}$, respectively.
Now, assume that $I_{2_{e}}=R_{2_{e}}$, then $I_{1_{e}}\not=R_{1_{e}}$ since
$J_{e}$ is a proper ideal of $R_{e}$. Let $R_{e}^{\prime
}=R_{e}/\{0\}_{1_{e}}\times R_{2_{e}}$, then $J_{e}^{\prime
}=J_{e}/\{0\}_{1_{e}}\times R_{2_{e}}$ is an $e$-$2$-absorbing primary ideal
of $R_{e}^{\prime }$. Now, since $R_{e}\cong R_{e}^{\prime }$ and $%
I_{1_{e}}\cong J_{e}^{\prime }$, $I_{1_{e}}$ is an $e$-$2$-absorbing primary
ideal of $R_{1_{e}}$. Similarly, if $I_{1_{e}}=R_{1_{e}}$, then $I_{2_{e}}$
is an $e$-$2$-absorbing primary ideal of $R_{2_{e}}$. Thus, assume that $%
I_{1_{e}}\not=R_{1_{e}}$, $I_{2_{e}}\not=R_{2_{e}}$ and $I_{1_{e}}$ is not
an $e$-primary ideal of $R_{1_{e}}.$ So, there exist $r_{1_{e}},s_{1_{e}}\in
R_{1_{e}}$ such that $r_{1_{e}}s_{1_{e}}\in I_{1_{e}}$ and neither $%
r_{1_{e}}\in I_{1_{e}}$ nor $s_{1_{e}}\in Gr(I_{1_{e}}).$ Now, let $%
x_{e}=(r_{1_{e}},1_{2_{e}}),$ $y_{e}=(1_{1_{e}},0_{2_{e}})$ and $%
z_{e}=(s_{1_{e}},1_{2_{e}}),$ hence $%
x_{e}y_{e}z_{e}=(r_{1_{e}}s_{1_{e}},0_{2_{e}})\in J_{e}$ but neither $%
x_{e}y_{e}=(r_{1_{e}},0_{2_{e}})\in J_{e}$ nor $%
x_{e}z_{e}=(r_{1_{e}}s_{1_{e}},1_{2_{e}})\in J_{e}$ nor $%
y_{e}z_{e}=(s_{1_{e}},0_{2_{e}})\in Gr(J_{e})=Gr(I_{1_{e}})\times
Gr(I_{2_{e}})$ which is a contradiction. Thus, $I_{1_{e}}$ is an $e$-primary
ideal of $R_{1_{e}}.$ Similarly, $I_{2_{e}}$ is an $e$-primary ideal of $%
R_{2_{e}}.$

$(ii)\Rightarrow (i)$ If $J_{e}=I_{1_{e}}\times R_{2_{e}}$ for some $e$-$2$%
-absorbing primary ideal $I_{1_{e}}$ of $R_{1_{e}}$ or $J_{e}=R_{1_{e}}%
\times I_{2_{e}}$ for some $e$-$2$-absorbing primary ideal $I_{2_{e}}$ of $%
R_{2_{e}}$, then it is clear that $J_{e}$ is an $e$-$2$-absorbing primary
ideal of $R_{e}$. Hence, assume that $J_{e}=I_{1_{e}}\times I_{2_{e}}$ for
some $e$-primary ideal $I_{1_{e}}$ of $R_{1_{e}}$ and some $e$-primary ideal
$I_{2_{e}}$ of $R_{2_{e}}$. Then $I_{1_{e}}^{\prime }=I_{1_{e}}\times
R_{2_{e}}$ and $I_{2_{e}}^{\prime }=R_{1_{e}}\times I_{2_{e}}$ are $e$%
-primary ideals of $R_{e}$. Hence $I_{1_{e}}^{\prime }\cap I_{2_{e}}^{\prime
}=I_{1_{e}}\times I_{2_{e}}=J_{e}$ is an $e$-$2$-absorbing primary ideal of $%
R_{e}$ by \textbf{Lemma 3.5}.

\end{proof}

\begin{theorem}
Let $R$ be a $G$-graded ring, $M$ a graded $R$%
-module, $N=\oplus _{g\in G}N_{g}$ a proper graded submodule of $M$ and $%
g\in G.$ Let $M_{g}$ be a $g$-multiplication. Then the following statements
are equivalent:
\begin{enumerate}[\upshape (i)]
   \item $N_{g}$ is a $g$-primary submodule of an $R_{e}$-module $M_{g}.$

    \item $(N_{g}:_{R_{e}}M_{g})$ is an $e$-primary ideal of $R_{e}.$

\end{enumerate}
  \end{theorem}
\begin{proof}
$(i)\Rightarrow (ii)$ Let $r_{e},s_{e}\in R_{e}$ such that $%
r_{e}s_{e}\in (N_{g}:_{R_{e}}M_{g})$ and $r_{e}\not\in
Gr((N_{g}:_{R_{e}}M_{g})).$ Thus, $r_{e}s_{e}M_{g}\subseteq N_{g}$ yields
that $s_{e}M_{g}\subseteq N_{g}$ and then $s_{e}\in (N_{g}:_{R_{e}}M_{g}).$
Therefore, $(N_{g}:_{R_{e}}M_{g})$ is an $e$-primary ideal of $R_{e}.$

$(ii)\Rightarrow (i)$ Let $r_{e}\in R_{e}$ and $m_{g}\in M_{g}$ such that $%
r_{e}m_{g}\in N_{g}$ and $r_{e}\not\in Gr((N_{g}:_{R_{e}}M_{g})).$ Hence, $%
r_{e}R_{e}m_{g}=r_{e}(R_{e}m_{g}:_{R_{e}}M_{g})M_{g}\subseteq N_{g}$ which
yields that $r_{e}(R_{e}m_{g}:_{R_{e}}M_{g})\subseteq (N_{g}:_{R_{e}}M_{g}).$
Thus, we get $(R_{e}m_{g}:_{R_{e}}M_{g})\subseteq (N_{g}:_{R_{e}}M_{g}),$ so
$m_{g}\in R_{e}m_{g}=(R_{e}m_{g}:_{R_{e}}M_{g})M_{g}\subseteq
(N_{g}:_{R_{e}}M_{g})M_{g}=N_{g}.$ Therefore, $N_{g}$ is a $g$-primary
submodule of an $R_{e}$-module $M_{g}.$
\end{proof}
\begin{theorem}
Let $R_{i}$ be a $G$-graded $gr$-Noetherian ring, $%
M_{i}$ a graded $R_{i}$-module, $N_{i}=\oplus _{g\in G}N_{i_{g}}$ a graded
submodule of $M_{i}$ for $i=1,2$ and $g\in G.$ Let $R=R_{1}\times R_{2}$ and
$M=M_{1}\times M_{2}$ such that $M_{g}$ is a fully $g$-idempotent $R_{e}$%
-module. Then:
\begin{enumerate}[\upshape (i)]
 \item $N_{1_{g}}$ is a $g$-$G2$-absorbing submodule of $M_{1_{g}}$ if and
only if $N_{1_{g}}\times M_{2_{g}}$ is a $g$-$G2$-absorbing submodule of $%
M_{g}$.

\item $N_{2_{g}}$ is a $g$-$G2$-absorbing submodule of $M_{2_{g}}$ if and
only if $M_{1_{g}}\times N_{2_{g}}$ is a $g$-$G2$-absorbing submodule of $%
M_{g}$.

\item If $N_{1_{g}}$ is a $g$-primary submodule of $M_{1_{g}}$ and $%
N_{2_{g}}$ is a $g$-primary submodule of $M_{2_{g}}$, then $N_{1_{g}}\times
N_{2_{g}}$ is a $g$-$G2$-absorbing submodule of $M_{g}$.
   \end{enumerate}
 \end{theorem}
\begin{proof}
$(i)$ Since $M_{g}$ is a fully $g$-idempotent $R_{e}$-module, $%
M_{i_{g}}$ is a fully $g$-idempotent $R_{i_{e}}$-module, for $i=1,2$, \textbf{by Lemma 3.4.} Now,
suppose that $N_{1_{g}}$ is a $g$-$G2$-absorbing submodule of $M_{1_{g}}$,
then by \textbf{Theorem 2.6} we get $(N_{1_{g}}:_{R_{1_{e}}}M_{1_{g}})$ is
an $e$-$2$-absorbing primary ideal of $R_{1_{e}}$. Thus, $(N_{1_{g}}\times
M_{2_{g}}:_{R_{e}}M_{g})=(N_{1_{g}}:_{R_{1_{e}}}M_{1_{g}})\times R_{2_{e}}$
is an $e$-$2$-absorbing primary ideal of $R_{e}$ by \textbf{Theorem 3.6}.
Hence, by \textbf{Theorem 3.2} we get $N_{1_{g}}\times M_{2_{g}}$ is a $g$-$G2$-absorbing submodule of $M_{g}$. Conversely, suppose that $%
N_{1_{g}}\times M_{2_{g}}$ is a $g$-$G2$-absorbing submodule of $M_{g}$,
then $(N_{1_{g}}\times
M_{2_{g}}:_{R_{e}}M_{g})=(N_{1_{g}}:_{R_{1_{e}}}M_{1_{g}})\times R_{2_{e}}$
is an $e$-$2$-absorbing primary ideal of $R_{e}$ by \textbf{Theorem 2.6}.
So, $(N_{1_{g}}:_{R_{1_{e}}}M_{1_{g}})$ is an $e$-$2$-absorbing primary
ideal of $R_{1_{e}}$ by \textbf{Theorem 3.6}. Thus, by \textbf{Theorem 3.2}
we get $N_{1_{g}}$ is a $g$-$G2$-absorbing submodule of $M_{1_{g}}$.

$(ii)$ The proof is similar to that in part $(i)$.

$(iii)$ Let $N_{i_{g}}$ be a $g$-primary submodule of $M_{i_{g}}$, then $%
(N_{i_{g}}:_{R_{i_{e}}}M_{i_{g}})$ is an $e$-primary ideal of $R_{i_{e}},$
for $i=1,2$. Now, since $(N_{1_{g}}\times
N_{2_{g}}:_{R_{e}}M_{g})=(N_{1_{g}}:_{R_{1_{e}}}M_{1_{g}})\times
(N_{2_{g}}:_{R_{2_{e}}}M_{2_{g}})$, $(N_{1_{g}}\times
N_{2_{g}}:_{R_{e}}M_{g})$ is an $e$-$2$-absorbing primary ideal of $R_{e}$
by \textbf{Theorem 3.6}. Therefore, $N_{1_{g}}\times N_{2_{g}}$ is a $g$-$G2$-absorbing submodule of $M_{g}$.
\end{proof}

\begin{theorem}
Let $R_{i}$ be a $G$-graded $gr$-Noetherian ring,
$M_{i}$ a graded $R_{i}$-module, $N_{i}=\oplus _{g\in G}N_{i_{g}}$ a graded
submodule of $M_{i}$ for $i=1,2$ and $g\in G.$ Let $R=R_{1}\times R_{2}$, $%
M=M_{1}\times M_{2}$ such that $M_{g}$ is a fully $g$-idempotent $R_{e}$%
-module and $N=N_{1}\times N_{2}$. Then the following statements are
equivalent:
\begin{enumerate}[\upshape (i)]
 \item $N_{g}$ is a $g$-$G2$-absorbing submodule of $M_{g}$.

 \item Either $N_{1_{g}}=M_{1_{g}}$ and $N_{2_{g}}$ is a $g$-$G2$-absorbing
submodule of $M_{2_{g}}$ or $N_{2_{g}}=M_{2_{g}}$ and $N_{1_{g}}$ is a $g$-$G2$-absorbing submodule of $M_{1_{g}}$ or $N_{1_{g}}$ and $N_{2_{g}}$ are $%
g $-primary submodules of $M_{1_{g}}$ and $M_{2_{g}}$, respectively.
 \end{enumerate}

\end{theorem}
\begin{proof}
$(i)\Rightarrow (ii)$ Assume that $N_{g}=N_{1_{g}}\times N_{2_{g}}$
is a $g$-$G2$-absorbing submodule of $M_{g}$. So, by \textbf{Theorem 2.6}
we get $(N_{g}:_{R_{e}}M_{g})=(N_{1_{g}}:_{R_{1_{e}}}M_{1_{g}})\times
(N_{2_{g}}:_{R_{2_{e}}}M_{2_{g}})$ is an $e$-$2$-absorbing primary ideal of $%
R_{e}$. Thus, we get either $(N_{1_{g}}:_{R_{1_{e}}}M_{1_{g}})=R_{1_{e}}$
and $(N_{2_{g}}:_{R_{2_{e}}}M_{2_{g}})$ is an $e$-$2$-absorbing primary
ideal of $R_{2_{e}}$ or $(N_{2_{g}}:_{R_{2_{e}}}M_{2_{g}})=R_{2_{e}}$ and $%
(N_{1_{g}}:_{R_{1_{e}}}M_{1_{g}})$ is an $e$-$2$-absorbing primary ideal of $%
R_{1_{e}}$ or $(N_{1_{g}}:_{R_{1_{e}}}M_{1_{g}})$ and $%
(N_{2_{g}}:_{R_{2_{e}}}M_{2_{g}})$ are $e$-primary ideals of $R_{1_{e}}$ and
$R_{2_{e}}$, respectively, by \textbf{Theorem 3.6}. Now, if $%
(N_{1_{g}}:_{R_{1_{e}}}M_{1_{g}})=R_{1_{e}}$ and $%
(N_{2_{g}}:_{R_{2_{e}}}M_{2_{g}})$ is an $e$-$2$-absorbing primary ideal of $%
R_{2_{e}}$, then $N_{1_{g}}=M_{1_{g}}$ and $N_{2_{g}}$ is a $g$-$G2$%
-absorbing submodule of $M_{2_{g}}$ by \textbf{Theorem 3.8}. Similarly, if $%
(N_{2_{g}}:_{R_{2_{e}}}M_{2_{g}})=R_{2_{e}}$ and $%
(N_{1_{g}}:_{R_{1_{e}}}M_{1_{g}})$ is an $e$-$2$-absorbing primary ideal of $%
R_{1_{e}}$, then $N_{2_{g}}=M_{2_{g}}$ and $N_{1_{g}}$ is a $g$-$G2$%
-absorbing submodule of $M_{1_{g}}$. If $(N_{i_{g}}:_{R_{i_{e}}}M_{i_{g}})$
is an $e$-primary ideal of $R_{i_{e}}$, then since $M_{i_{g}}$ is a $g$%
-multiplication $R_{i_{e}}$-module, $N_{i_{g}}$ is a $g$-primary submodule
of $M_{i_{g}}$, for $i=1,2,$ \textbf{by Theorem 3.7}.

$(ii)\Rightarrow (i)$ Clearly, by \textbf{Theorem 3.8}.
\end{proof}

\begin{theorem}
Let $R$ be a $G$-graded ring, $I=\oplus _{g\in
G}I_{g}$ be a graded ideal of $R$ and $S_{e}\subseteq R_{e}$ be a
multiplicatively closed subset. Then the set $\Gamma =\{J_{e}|$ $J_{e}$ is
an ideal of $R_{e}$, $S_{e}\cap J_{e}=\emptyset $, $I_{e}\subseteq J_{e}\}$
has a maximal element and such maximal elements are $e$-prime ideals of $%
R_{e}$.
\end{theorem}
\begin{proof}
Since $I_{e}\in \Gamma $, $\Gamma \not=\emptyset $. The set $\Gamma
$ is a partially ordered set with respect to set inclusion $^{\prime \prime
}\subseteq ^{\prime \prime }$. Now, let $\Delta $ be a totally ordered
subset of $\Gamma $, then $J=\cup _{J_{e}\in \Delta }J_{e}$ is an ideal of $%
R_{e}$. Now, let $P_{e}$ be a maximal element of $\Gamma $ and $%
r_{e},s_{e}\in R_{e}$ such that $r_{1_{e}}\not\in P_{e}$ and $%
r_{2_{e}}\not\in P_{e}$. Thus, $P_{e}\subsetneq (P_{e}+\langle
r_{1_{e}}\rangle )$, which conclude that $(P_{e}+\langle r_{1_{e}}\rangle
)\cap S_{e}\not=\emptyset ,$ so there exists $s_{1_{e}}\in S_{e}$ such that $%
s_{1_{e}}=p_{1_{e}}+r_{1_{e}}t_{1_{e}}$ where $p_{1_{e}}\in P_{e}$ and $%
t_{1_{e}}\in R_{e}$. Similarly, there exists $s_{2_{e}}\in S_{e}$ such that $%
s_{2_{e}}=p_{2_{e}}+r_{2_{e}}t_{2_{e}}$ where $p_{2_{e}}\in P_{e}$ and $%
t_{2_{e}}\in R_{e}$. Hence, $%
s_{1_{e}}s_{2_{e}}=(p_{1_{e}}+r_{1_{e}}t_{1_{e}})(p_{2_{e}}+r_{2_{e}}t_{2_{e}})=p_{1_{e}}p_{2_{e}}+p_{1_{e}}r_{2_{e}}t_{2_{e}}+p_{2_{e}}r_{1_{e}}t_{1_{e}}+r_{1_{e}}r_{2_{e}}t_{1_{e}}t_{2_{e}}\in S_{e}\backslash P_{e}
$, which yields that $r_{1_{e}}r_{2_{e}}t_{1_{e}}t_{2_{e}}\not\in P_{e}$ and
then $r_{1_{e}}r_{2_{e}}\not\in P_{e}$. Therefore, $P_{e}$ is an $e$-prime
ideal of $R_{e}$.
\end{proof}
\begin{theorem}
 Let $R$ be a $G$-graded ring and $I=\oplus _{g\in
G}I_{g},$ $P=\oplus _{g\in G}P_{g}$ be two graded ideals of $R.$ Let $P_{e}$
be an $e$-prime ideal of $R_{e}$ such that $I_{e}\subseteq P_{e}$. Then the
following statements are equivalent:
\begin{enumerate}[\upshape (i)]
 \item $P_{e}$ is a minimal $e$-prime ideal of $R_{e}$ over $I_{e}$.

 \item  $R_{e}\backslash P_{e}$ is a multiplicatively closed subset of $R_{e}$
that is maximal with respect to missing $I_{e}$.

 \item  For each $x_{e}\in P_{e}$, there exists $y_{e}\in R_{e}\backslash
P_{e}$ and nonnegative integer $n$ such that $y_{e}x_{e}^{n}\in I_{e}$.

 \end{enumerate}

\end{theorem}
\begin{proof}
$(i)\Rightarrow (ii)$ Suppose that $P_{e}$ is a minimal $e$-prime
ideal of $R_{e}$ over $I_{e}$. Now, let $S_{e}=R_{e}\backslash P_{e}$, then $%
S_{e}$ is a multiplicatively closed subset of $R_{e}$ and there exists a
maximal element in the set of ideals of $R_{e}$ containing $I_{e}$ and
disjoint from $S_{e}$. Assume that $J_{e}$ is a maximal then $J_{e}$ is an $%
e $-prime ideal of $R_{e}$ by \textbf{Theorem 3.10}. Since $P_{e}$ is a
minimal, $P_{e}=J_{e}$ and so $S_{e}$ is a maximal with respect to missing $%
I_{e}$.

$(ii)\Rightarrow (iii)$ Let $0\not=x_{e}\in P_{e}$ and $S_{e}=%
\{y_{e}x_{e}^{n}|y_{e}\in R_{e}\backslash P_{e},$ $n=0,1,2,...\}$. Then $%
R_{e}\backslash P_{e}\subsetneq S$. Since $R_{e}\backslash P_{e}$ is a
maximal, there exists $y_{e}\in R_{e}\backslash P_{e}$ and $n\in
\mathbb{Z}
^{+}$ such that $y_{e}x_{e}^{n}\in I_{e}$.

$(iii)\Rightarrow (i)$ Assume that $I_{e}\subset J_{e}\subseteq P_{e}$,
where $J_{e}$ is an $e$-prime ideal of $R_{e}$. If there exists $x_{e}\in
P_{e}\backslash J_{e}$, then there exists $y_{e}\in R_{e}\backslash P_{e}$
and $n\in
\mathbb{Z}
^{+}$\ such that $y_{e}x_{e}^{n}\in I_{e}\subseteq J_{e}$. But $y_{e}\not\in
J_{e}$, so $x_{e}^{n}\in J_{e}$, a contradiction. Therefore, $P_{e}$ is a
minimal $e$-prime ideal of $R_{e}$ over $I_{e}$.
\end{proof}


\begin{theorem}
Let $R$ be a $G$-graded ring and $I=\oplus _{g\in
G}I_{g}$ a graded ideal of $R.$ If $I_{e}$ is an $e$-$2$-absorbing ideal of $%
R_{e}$, then there are at most two $e$-prime ideals of $R_{e}$ that are
minimal over $I_{e}$.
\end{theorem}
\begin{proof}
 Let $\Gamma =\{P_{i_{e}}$ $|$ $P_{i_{e}}$ is an $e$-prime ideal of $%
R_{e}$ that is minimal over $I_{e}\}.$ Suppose that $\Gamma $ has at least
three elements. Let $P_{1_{e}},P_{2_{e}}\in \Gamma $ be two distinct $e$%
-prime ideals. Hence, there exist $x_{1_{e}}\in P_{1_{e}}\backslash P_{2_{e}}
$ and $x_{2_{e}}\in P_{2_{e}}\backslash P_{1_{e}}$. First we show that $%
x_{1_{e}}x_{2_{e}}\ \in I_{e}$. By \textbf{Theorem 3.11}, there exist $%
c_{2_{e}}\in R_{e}\backslash P_{1_{e}}$ and $c_{1_{e}}\in R_{e}\backslash
P_{2_{e}}$ such that $c_{2_{e}}x_{1_{e}}^{k_{1}}\in I_{e}$ and $%
c_{1_{e}}x_{2_{e}}^{k_{2}}\in I_{e}$ for some $k_{1},k_{2}\in
\mathbb{Z}
^{+}$. Since $x_{1_{e}},x_{2_{e}}\not\in P_{1_{e}}\cap P_{2_{e}}$ and $I_{e}$
is an $e$-$2$-absorbing ideal of $R_{e}$, we conclude that $%
c_{2_{e}}x_{1_{e}}\in I_{e}$ and $c_{1_{e}}x_{2_{e}}\in I_{e}$. Since $%
x_{1_{e}},x_{2_{e}}\not\in P_{1_{e}}\cap P_{2_{e}}$ and $%
c_{2_{e}}x_{1_{e}},c_{1_{e}}x_{2_{e}}\in I_{e}\subseteq P_{1_{e}}\cap
P_{2_{e}}$, we conclude that $c_{1_{e}}\in P_{1_{e}}\backslash P_{2_{e}}$
and $c_{2_{e}}\in P_{2_{e}}\backslash P_{1_{e}}$, and thus $%
c_{1_{e}},c_{2_{e}}\not\in P_{1_{e}}\cap P_{2_{e}}$. Since $%
c_{2_{e}}x_{1_{e}}\in I_{e}$ and $c_{1_{e}}x_{2_{e}}\in I_{e}$, we have $%
(c_{1_{e}}+c_{2_{e}})x_{1_{e}}x_{2_{e}}\in I_{e}$. Hence, $%
c_{1_{e}}+c_{2_{e}}\not\in P_{1_{e}}$ and $c_{1_{e}}+c_{2_{e}}\not\in
P_{2_{e}}$. Since $(c_{1_{e}}+c_{2_{e}})x_{1_{e}}\not\in P_{2_{e}}$ and $%
(c_{1_{e}}+c_{2_{e}})x_{2_{e}}\not\in P_{1_{e}}$, neither $%
(c_{1_{e}}+c_{2_{e}})x_{1_{e}}\in I_{e}$ nor $(c_{1_{e}}+c_{2_{e}})x_{2_{e}}%
\in I_{e}$, and hence $x_{1_{e}}x_{2_{e}}\in I_{e}$. Now, suppose that there
exists $P_{3_{e}}\in \Gamma $ such that $P_{3_{e}}$ is neither $P_{1_{e}}$
nor $P_{2_{e}}$. Then we can choose $y_{1_{e}}\in P_{1_{e}}\backslash
(P_{2_{e}}\cup P_{3_{e}})$, $y_{2_{e}}\in P_{2_{e}}\backslash (P_{1_{e}}\cup
P_{3_{e}})$, and $y_{3_{e}}\in P_{3_{e}}\backslash (P_{1_{e}}\cup P_{2_{e}})$%
. By the previous argument $y_{1_{e}}y_{2_{e}}\in I_{e}$. Since $%
I_{e}\subseteq P_{1_{e}}\cap P_{2_{e}}\cap P_{3_{e}}$ and $%
y_{1_{e}}y_{2_{e}}\in I_{e}$, we conclude that either $y_{1_{e}}\in P_{3_{e}}
$ or $y_{2_{e}}\in P_{3_{e}},$ a contradiction. Hence, $\Gamma $ has at most
two elements.
\end{proof}

\begin{theorem}
Let $R$ be a $G$-graded ring and $I=\oplus _{g\in
G}I_{g}$ a graded ideal of $R$ such that $I_{e}$ is an $e$-$2$-absorbing
primary ideal of $R_{e}$. Then one of the following statements must hold.
\begin{enumerate}[\upshape (i)]
 \item $Gr(I_{e})=P_{e},$ where $P_{e}$ is an $e$-prime ideal of $R_{e}$.

 \item $Gr(I_{e})=P_{1_{e}}\cap P_{2_{e}}$, where $P_{1_{e}}$ and $P_{2_{e}}$
are the only distinct $e$-prime ideals of $R_{e}$ that are minimal over $%
I_{e}$.
 \end{enumerate}

\end{theorem}
\begin{proof}
Suppose that $I_{e}$ is an $e$-$2$-absorbing primary ideal of $R_{e}
$, then $Gr(I_{e})$ is an $e$-$2$-absorbing ideal of $R_{e}$ by \textbf{%
Theorem 2.7}. Hence, since $Gr(Gr(I_{e}))=Gr(I_{e})$, the result follows
from \textbf{Theorem 3.12}.
\end{proof}

\begin{theorem}
Let $R$ be $G$-graded $gr$-Noetherian ring, $M$
a graded $R$-module, $N=\oplus _{g\in G}N_{g}$ a graded submodule of $M$ and
$g\in G.$ If $N_{g}$ is a $g$-$G2$-absorbing submodule of an $R_{e}$-module
$M_{g}$ and $m_{g}\in M_{g}\backslash N_{g}$, then either $%
Gr((N_{g}:_{R_{e}}m_{g}))$ is an $e$-prime ideal of $R_{e}$ or there exist $%
r_{e}\in R_{e}\backslash Gr((N_{g}:_{R_{e}}m_{g}))$ and $n\in
\mathbb{Z}
^{+}$ such that $Gr((N_{g}:_{R_{e}}r_{e}^{n}m_{g}))$ is an $e$-prime ideal
of $R_{e}$.
\end{theorem}
\begin{proof}
Let $m_{g}\in M_{g}\backslash N_{g}$ and assume that $N_{g}$ is a $%
g $-$G2$-absorbing submodule of an $R_{e}$-module $M_{g},$\ then $%
Gr((N_{g}:_{R_{e}}M_{g}))$ is an $e$-$2$-absorbing ideal of $R_{e}$ by
\textbf{Corollary 2.8}. Now, by\textbf{\ Theorem 3.13} we get either $%
Gr((N_{g}:_{R_{e}}M_{g}))=P_{1_{e}}$ or $Gr((N_{g}:_{R_{e}}M_{g}))=P_{1_{e}}%
\cap P_{2_{e}}$, where $P_{1_{e}}$ and $P_{2_{e}}$ are distinct $e$-prime
ideals of $R_{e}$. First, assume that $Gr((N_{g}:_{R_{e}}M_{g}))=P_{1_{e}}$.
Let $r_{e}s_{e}\in Gr((N_{g}:_{R_{e}}m_{g}))$ for some $r_{e},s_{e}\in R_{e}$%
, then $(r_{e}s_{e})^{n}m_{g}\in N_{g}$ for some $n\in
\mathbb{Z}
^{+}$. Thus, we get either $r_{e}^{nt}m_{g}\in N_{g}$ or $s_{e}^{nk}m_{g}\in
N_{g}$ for some $t,k\in
\mathbb{Z}
^{+}$ or $(r_{e}s_{e})^{n}\in (N_{g}:_{R_{e}}M_{g})$ as $N_{g}$ is a $g$-$G2 $-absorbing submodule of $M_{g}$. If either $r_{e}^{nt}m_{g}\in N_{g}$
or $s_{e}^{nk}m_{g}\in N_{g}$, then either $r_{e}\in
Gr((N_{g}:_{R_{e}}m_{g}))$ or $s_{e}\in Gr((N_{g}:_{R_{e}}m_{g}))$. Now, if $%
(r_{e}s_{e})^{n}\in (N_{g}:_{R_{e}}M_{g})$, then $r_{e}s_{e}\in P_{1_{e}}$.
Since $P_{1_{e}}$ is an $e$-prime ideal of $R_{e}$, then either $r_{e}\in
P_{1_{e}}=Gr((N_{g}:_{R_{e}}M_{g}))\subseteq Gr((N_{g}:_{R_{e}}m_{g}))$ or $%
s_{e}\in P_{1_{e}}=Gr((N_{g}:_{R_{e}}M_{g}))\subseteq
Gr((N_{g}:_{R_{e}}m_{g}))$. Therefore, $Gr((N_{g}:_{R_{e}}m_{g}))$ is an $e$%
-prime ideal of $R_{e}$. Now, we can assume that $%
Gr((N_{g}:_{R_{e}}M_{g}))=P_{1_{e}}\cap P_{2_{e}}$. If $P_{1_{e}}\subseteq
Gr((N_{g}:_{R_{e}}m_{g}))$, then we get the result by using the previous
argument. If $P_{1_{e}}\not\subseteq Gr((N_{g}:_{R_{e}}m_{g}))$, then there
exists $r_{e}\in P_{1_{e}}\backslash Gr((N_{g}:_{R_{e}}m_{g}))$. Also, $%
P_{1_{e}}P_{2_{e}}\subseteq Gr(P_{1_{e}}P_{2_{e}})=Gr(P_{1_{e}}\cap
P_{2_{e}})=Gr((N_{g}:_{R_{e}}M_{g}))\subseteq Gr((N_{g}:_{R_{e}}m_{g}))$.
Now, $R$ is $gr$-Noetherian implies that $R_{e}$ is Noetherian, so there exists $%
n\in
\mathbb{Z}
^{+}$ such that $(P_{1_{e}}P_{2_{e}})^{n}m_{g}\subseteq N_{g}$ and then $%
r_{e}^{n}P_{2_{e}}^{n}m_{g}\subseteq N_{g}$. Hence, $P_{2_{e}}\subseteq
Gr((N_{g}:_{R_{e}}r_{e}^{n}m_{g}))$ and we get the result by a similar
argument.
\end{proof}
Let $N=\oplus _{g\in G}N_{g}$ be a graded submodule of a graded $R
$-module $M$ and $g\in G.$ An efficient covering of $N_{g}$ is a covering $%
N_{g}\subseteq N_{1_{g}}\cup N_{2_{g}}\cup ...\cup N_{n_{g}}$ where $%
N_{i_{g}}$ is a submodule of an $R_{e}$-module $M_{g}$ such that $N_{g}$ $%
\not\subseteq $ $N_{i_{g}}$ for $i=1,...,n$. Also, we say that $%
N_{g}=N_{1_{g}}\cup N_{2_{g}}\cup ...\cup N_{n_{g}}$ is an efficient union
if none of the $N_{i_{g}}$ may be excluded, (see \cite{23}).

\begin{theorem}
Let $R$ be a $G$-graded $gr$-Noetherian ring, $M$
a graded $R$-module, $N=\oplus _{g\in G}N_{g}$ and $N_{i}=\oplus _{g\in
G}N_{i_{g}}$ be graded submodules of $M,$ for $i=1,...,n$ and $g\in G.$ Let $%
N_{g}\subseteq N_{1_{g}}\cup N_{2_{g}}\cup ...\cup N_{n_{g}}$ be an
efficient covering consisting of submodules of an $R_{e}$-module $M_{g}$,
where $n>2$. If $Gr((N_{i_{g}}:_{R_{e}}M_{g}))\not\subseteq
Gr((N_{k_{g}}:_{R_{e}}m_{k_{g}}))$ for all $m_{k_{g}}\in M_{g}\backslash
N_{k_{g}}$ whenever $i\not=k$, then $N_{i_{g}}$ is not a $g$-$G2$-absorbing
submodule of $M_{g},$ for $i=1,...,n$.
\end{theorem}
\begin{proof}
Suppose that there exists $k\in \{1,...,n\}$ such that $N_{k_{g}}$
is a $g$-$G2$-absorbing submodule of $M_{g}$. Since $N_{g}\subseteq
N_{1_{g}}\cup N_{2_{g}}\cup ...\cup N_{n_{g}}$ is an efficient covering of $%
N_{g}$, $N_{g}\not\subseteq N_{k_{g}}$, so there exists $m_{k_{g}}\in
N_{g}\backslash N_{k_{g}}$. It is easy to see that $N_{g}=(N_{1_{g}}\cap
N_{g})\cup (N_{2_{g}}\cap N_{g})\cup ...\cup (N_{n_{g}}\cap N_{g})$ is an
efficient union of $N_{g}$, so by \cite[Lemma 2.2]{23}, we have $\cap
_{i\not=k}(N_{g}\cap N_{i_{g}})=\cap _{i=1}^{n}(N_{g}\cap
N_{i_{g}})\subseteq N_{g}\cap N_{k_{g}}$. Now, by using \textbf{Theorem 3.14}%
, we get either $Gr((N_{k_{g}}:_{R_{e}}m_{k_{g}}))$ is an $e$-prime ideal of
$R_{e}$ or there exists $r_{e}\in R_{e}\backslash
Gr((N_{k_{g}}:_{R_{e}}m_{k_{g}}))$ such that $%
Gr((N_{k_{g}}:_{R_{e}}r_{e}^{t}m_{k_{g}}))$ is an $e$-prime ideal of $R_{e}$%
, where $t\in
\mathbb{Z}
^{+}$. Hence, if $Gr((N_{k_{g}}:_{R_{e}}m_{k_{g}}))$ is an $e$-prime ideal
of $R_{e}$, then $Gr((N_{i_{g}}:_{R_{e}}M_{g}))\not\subseteq
Gr((N_{k_{g}}:_{R_{e}}m_{k_{g}}))$ for all $i\not=k$. So, there exists $%
s_{i_{e}}\in Gr((N_{i_{g}}:_{R_{e}}M_{g}))\backslash
Gr((N_{k_{g}}:_{R_{e}}m_{k_{g}}))$, where $i\not=k$. Thus, $%
s_{i_{e}}^{j_{i}}\in (N_{i_{g}}:_{R_{e}}M_{g})\backslash
(N_{k_{g}}:_{R_{e}}m_{k_{g}})$ where $i\not=k$ and $j_{i}\in
\mathbb{Z}
^{+}$. Let $s_{e}=\Pi _{i\not=k}s_{i_{e}}$ and $j=max%
\{j_{1},j_{2},...,j_{k-1},j_{k+1},...,j_{n}\}$, then $s_{e}^{j}\in
(N_{i_{g}}:_{R_{e}}M_{g})\backslash (N_{k_{g}}:_{R_{e}}m_{k_{g}})$ for all $%
i\not=k$. Thus, $s_{e}^{j}m_{k_{g}}\in \cap _{i\not=k}(N_{g}\cap
N_{i_{g}})\backslash (N_{g}\cap N_{k_{g}})=\emptyset $, a contradiction.
Therefore, $N_{k_{g}}$ is not a $g$-$G2$-absorbing submodule of $M_{g}$.
Now, if $Gr((N_{k_{g}}:_{R_{e}}r_{e}^{t}m_{k_{g}}))$ is an $e$-prime ideal
of $R_{e}$, where $r_{e}\in R_{e}\backslash Gr((N_{k_{g}}:_{R_{e}}m_{k_{g}}))$
and $t\in
\mathbb{Z}
^{+}$. So, there exists $s_{i_{e}}\in
Gr((N_{i_{g}}:_{R_{e}}M_{g}))\backslash
Gr((N_{k_{g}}:_{R_{e}}r_{e}^{t}m_{k_{g}}))$, where $i\not=k$. Hence, $%
s_{i_{e}}^{j_{i}}\in (N_{i_{g}}:_{R_{e}}M_{g})\backslash
(N_{k_{g}}:_{R_{e}}r_{e}^{t}m_{k_{g}})$ where $i\not=k$ and $j_{i}\in
\mathbb{Z}
^{+}.$ Let $s_{e}=\Pi _{i\not=k}s_{i_{e}}$ and $j=max%
\{j_{1},j_{2},...,j_{k-1},j_{k+1},...,j_{n}\}$, then $s_{e}^{j}\in
(N_{i_{g}}:_{R_{e}}M_{g})\backslash (N_{k_{g}}:_{R_{e}}r_{e}^{t}m_{k_{g}})$
for all $i\not=k$. \ Therefore, $s_{e}^{j}r_{e}^{t}m_{k_{g}}\in \cap
_{i\not=k}(N_{g}\cap N_{i_{g}})\backslash (N_{g}\cap N_{k_{g}})=\emptyset $,
a contradiction. Therefore, $N_{k_{g}}$ is not a $g$-$G2$-absorbing
submodule of $M_{g}$.
\end{proof}

\begin{theorem}
Let $R$ be a $G$-graded $gr$-Noetherian ring, $M$ a
graded $R$-module, $N=\oplus _{g\in G}N_{g}$ and $N_{i}=\oplus _{g\in
G}N_{i_{g}}$ be graded submodules of $M,$ for $i=1,...,n$ where $n\geqslant
2 $ and $g\in G$. If $N_{g}\subseteq N_{1_{g}}\cup N_{2_{g}}\cup ...\cup
N_{n_{g}}$ such that at most two of $N_{1_{g}},N_{2_{g}},...,N_{n_{g}}$ are
not $g$-$G2$-absorbing submodules of an $R_{e}$-module\ $M_{g}$ and $%
Gr((N_{i_{g}}:_{R_{e}}M_{g}))\not\subseteq Gr((N_{k_{g}}:_{R_{e}}m_{k_{g}})$
for all $m_{k_{g}}\in M_{g}\backslash N_{k_{g}}$ whenever $i\not=k$, then $%
N_{g}\subseteq N_{i_{g}}$ for some $i=1,...,n$.
\end{theorem}
\begin{proof}
If $n=2$, then we get the result since if a graded submodule is contained
in a union of two graded submodules, then it is contained in one of them. Now, take
$n>2$ and $N_{g}\not\subseteq N_{i_{g}}$ for all $i=1,...,n$, then $%
N_{g}\subseteq N_{1_{g}}\cup N_{2_{g}}\cup ...\cup N_{n_{g}}$ is an
efficient covering of $N_{g}$. So, by \textbf{Theorem 3.15}, $N_{i_{g}}$ is
not a $g$-$G2$-absorbing submodule for all $i=1,...,n$, a contradiction.
Hence, $N_{g}\subseteq N_{i_{g}}$ for some $i=1,...,n$.
\end{proof}


\bigskip\bigskip\bigskip\bigskip

\end{document}